\documentclass[11pt,reqno]{amsart}
\usepackage{graphicx}
\usepackage{float}
\usepackage{amsfonts}
\usepackage[caption = false]{subfig}
\usepackage{amsmath}
\usepackage{graphics}
\usepackage{float}
\usepackage{enumerate}
\usepackage{mathtools}
\usepackage{amsthm}
\usepackage[english]{babel}
\usepackage{amsfonts,amssymb}
\usepackage{mathrsfs}
\usepackage[utf8x]{inputenc}\usepackage[english]{babel}
\usepackage{soul}
\usepackage[all]{xy,xypic}
\usepackage{cite}
\usepackage{relsize}
\numberwithin{equation}{section}
\newtheorem{theorem}{Theorem}[section]
\newtheorem{corollary}[theorem]{Corollary}
\newtheorem{lemma}[theorem]{Lemma}

\theoremstyle{definition}
\newtheorem{definition}[theorem]{Definition}
\theoremstyle{remark}
\newtheorem{remark}[theorem]{Remark}
\newtheorem{example}{Example}

\numberwithin{equation}{section}
\DeclareMathOperator{\RE}{Re}
\DeclareMathOperator{\IM}{Im}

\begin{document}
	
	\title[\tiny{On a Subclass of Starlike Functions Associated with a Strip Domain}]{On a Subclass of Starlike Functions Associated with a Strip Domain}
	
	\author[S. S. Kumar]{S. Sivaprasad Kumar}
	\address{Department of Applied Mathematics, Delhi Technological University, Delhi--110042, India}
	\email{spkumar@dce.ac.in}

    	\author[N. Verma]{Neha Verma*}
	\address{Department of Applied Mathematics, Delhi Technological University, Delhi--110042, India}
	\email{nehaverma1480@gmail.com}

	
	\maketitle
	\begin{abstract}
In the present investigation, we introduce a new subclass of starlike functions defined by 
$\mathcal{S}^{*}_{\tau}:=\{f\in \mathcal{A}:zf'(z)/f(z) \prec 1+\arctan z=:\tau(z)\}$, where $\tau(z)$ maps the unit disk $\mathbb {D}:= \{z\in \mathbb{C}:|z|<1\}$ onto a strip domain.
We derive structural formulae, growth and distortion theorems for $\mathcal{S}^{*}_{\tau}$. Also, inclusion relations with some well-known subclasses of $\mathcal{S}$ are established and obtain sharp radius estimates, as well as sharp coefficient bounds for the initial five coefficients and the second and third order Hankel determinants of $\mathcal{S}^{*}_{\tau}$.
\end{abstract}
\maketitle
	
\section{Introduction}
\label{intro}

Let $\mathcal {A}$ denote the class of normalized analytic functions $f(z)$ defined in the open unit disk $\mathbb {D}:= \{z\in \mathbb{C}:|z|<1\}$, having the form
\begin{equation}
f(z) = z+\sum_{n=2}^{\infty}a_nz^n\label{form}
\end{equation}and $\mathcal {S}$ be the subclass of $\mathcal {A}$ consisting of univalent functions. Let us consider a subclass $\mathcal{P}$ consisting of functions with positive real part having the series expansion of the form $p(z)=1+\sum_{n=1}^{\infty}p_n z^n$. Let $g,h$ be two analytic functions and $\omega$ be a Schwarz function satisfying $\omega(0)=0$ and $|\omega(z)|\leq |z|$ such that $g(z)=h(\omega(z))$, then $g$ is subordinate to $h$, or $g(z)\prec h(z)$. Let $\mathcal {S}^{*}$ be the subclass of $\mathcal {S}$ containing starlike functions, mapping the unit disk $\mathbb {D}$ onto the starlike domain. Geometrically, we say that $f\in \mathcal{S}^{*}$ if $\RE (zf'(z)/f(z))>0$. 

In 1992, Ma and Minda \cite{ma-minda} introduced the unified classes of starlike and convex functions using subordination, given as follows:
\begin{equation*}
\mathcal {S}^*(\psi)=\Bigg\{f\in \mathcal {S}:\dfrac{z f'(z)}{f(z)}\prec \psi(z)\Bigg\}\quad \text{and} \quad \mathcal {C}(\psi) =\Bigg\{f\in \mathcal {S} :1+\dfrac{zf''(z)}{f'(z)}\prec \psi (z)\Bigg\}
\end{equation*}

where $\psi$ is an analytic univalent function having positive real part, $\psi(\mathbb{D})$ is starlike with respect to $\psi(0)=1$, symmetric about the real axis and $\psi'(0)>0$. Specializing $\psi$, yields well-known subclasses of $\mathcal{S}^*$.
Many authors discussed the geometrical properties, radius results and coefficient estimations for the starlike class $\mathcal{S}^*(\psi)$. Janowski \cite{1janowski} introduced the subclass, $\mathcal{S}^*[A,B]:=\mathcal{S}^*((1+Az)/(1+Bz))$ where $-1\leq B<A\leq 1$. In 1996, Sok\'{o}\l \ and Stankiewicz \cite{stan} studied the class $\mathcal{S}^{*}_ L:=\mathcal{S}(\sqrt{1+z})$, while Sharma et al. \cite{sharma} proposed the class $\mathcal{S}^{*}_C:=\mathcal{S}(1+4z/3+2z^2/3)$. The classes $\mathcal{S}^{*}_e:=\mathcal{S}^{*}(e^z)$ and $\Delta^*:=\mathcal{S}^{*}(z+\sqrt{1+z^2})$ were introduced by Mendiratta et al.\cite{mendi} and Raina et al.\cite{raina}, respectively. Also, Goel and Kumar \cite{goel} considered the class $\mathcal{S}_{SG}^{*}:=\mathcal{S}^{*}(2/(1+e^{-z}))$ whereas Kumar and Kamaljeet \cite{kumar-ganganiaCardioid-2021} investigated the class $\mathcal{S}_{\wp}^{*}:=\mathcal{S}^{*}(1+ze^z)$ in their study.

In the past, a significant work has been done on strip domains, see \cite{kwon2014,kwon2019,mahzoon2023}. For instance, in 2012, Kuroki and Owa \cite{kuroki2012} defined a strip domain related class $\mathcal{S}(\alpha,\beta)$ of analytic functions $f\in \mathcal{A}$ such that $\alpha<\RE(zf'(z)/f(z))<\beta$ for real numbers $\alpha$ and $\beta$ so that $\alpha<1<\beta$. Later, in 2017, Kargar et al. \cite{kargar2017} introduced another strip domain related class $\mathcal{M}(\alpha)$ given by
\begin{equation*}
    \mathcal{M}(\alpha)=\bigg\{f\in \mathcal{A}:1+ \frac{\alpha-\pi}{2\sin \alpha}<\RE\bigg(\frac{zf'(z)}{f(z)}\bigg)<1+\frac{\alpha}{2\sin \alpha};\pi/2\leq \alpha<\pi\bigg\},
\end{equation*}
which was studied by Sun et al. \cite{sun2019} as well. Motivated by these works, we introduce a new subclass of starlike functions associated with a transcendental function $\psi(z) = 1 + \arctan z$, that maps the unit disk onto a vertical strip domain. Exploring strip domains with a transcendental function is a challenge due to computational complexities, and until now, no attempts have been made in this direction. However, this paper successfully navigates these difficulties, yielding several sharp radius and coefficient related results. 

	
\section{Certain Properties of $\tau(z):=1+\arctan z$}
\noindent Suppose that $\psi(z)=1+\arctan z=:\tau (z)$ and it maps the unit disk onto a domain denoted by $\Omega_\tau$. Now, we will study some geometric properties of this function $\tau(z)$ in the coming results. Before that, a lemma is presented below to determine the range for the real part of $\tau(z)$, which is useful for establishing radius problems and inclusion relations in the coming sections.

\begin{lemma}
 \label{lemma1}
    For $0<r<1,$ the function $\tau(z)$ satisfies 
    $$\min_{|z|=r}\RE \tau(z)=\tau(-r) \quad \text{and} \quad \max_{|z|=r}\RE \tau(z)=\tau(r).$$ 
\end{lemma}
\begin{proof}
Let $\tau (z)=1+\arctan z$. For $\mathbb {D}_{r_{0}}=\{z\in \mathbb{C}:|z|<r_{0}\},$ a boundary point of $\tau(\mathbb{D}_{r_{0}})$ can be expressed in the form $\tau(r_{0}e^{i \theta})=1+\arctan(r_{0} e^{i \theta}).$ As $\eta \tau'(\eta)$ represents the outward normal at the point $\tau(\eta),$ where $|\eta|=r_0.$ So, for $\eta=r_0 e^{i \theta},$ we have
$$\eta \tau'(\eta)=\dfrac{r_0 e^{i \theta}}{1+(r_{0} e^{i \theta})^2}.$$
In order to find the bounds for the real part, it suffices to find the points at which the imaginary part of the normal, is a constant. Thus, 
$$\IM (\eta \tau'(\eta))=t(\theta):=\zeta(r_0 \sin \theta+{r_0}^3 \cos(2 \theta) \sin \theta+{r_0}^3 \sin(2 \theta) \cos \theta),$$ where $$\zeta=\dfrac{1}{|1+(r_{0} e^{i \theta})^2|^2}.$$
We observe that $t(\theta)=0$ for $\theta=0$ and $2 \pi.$ By taking $r_0=1,$ a simple computation shows that the minimum and maximum of $\tau$ are obtained at $\theta=\pi$ and 0 respectively. Thus, we obtain  
\begin{equation*}
\min_{|z|=r}\RE \tau(z)=1-\arctan r \quad \text{and} \quad \max_{|z|=r}\RE \tau(z)=1+\arctan r
\end{equation*}
as $r_0$ is chosen arbitrarily.
\end{proof}

Now, we discuss some geometrical properties of the function $\tau(z)$ and $\tau(\mathbb{D})=\Omega_\tau$.

\begin{theorem}\label{1Fconvex}
 The function $\tau(z)=1+\arctan z$ is a convex univalent function.
\end{theorem}
\begin{proof}
Let $g(z):=\arctan z.$ Note that $g(0)=0$ and 
\begin{equation*}
    g'(z)=\frac{1}{1+z^2}\prec \frac{1}{1+z}\in \mathcal{P},
\end{equation*}
which implies that $\RE g'(z)>0.$ Thus $g(z)$ is a univalent function. Furthermore 
\begin{equation*}
    1+\dfrac{zg''(z)}{g'(z)}=\dfrac{1-z^2}{1+z^2}\prec \dfrac{1-z}{1+z}\in \mathcal{P}.
\end{equation*}
Therefore $\RE (1+zg''(z)/g'(z))>0$. Consequently, both $g$ and $\tau$ are convex univalent functions.
\end{proof}

\begin{remark}\label{1Symm-Caratheodory}
Since $\tau'(0)>0$ and $\tau(\bar{z})=\overline{\tau(z)},$ the domain $\Omega_{\tau}=\tau(\mathbb{D})$ is symmetric about the real axis. Using Theorem~\ref{1Fconvex}, we deduce that $\RE{\tau(z)}\geq 1-\pi/4>0$, which means that $\tau(z)$ is a Carath\'{eodory} function.
\end{remark}

\begin{remark}
By utilizing an alternate proof technique, we can see that the conclusion of Lemma \ref{lemma1} can be directly derived from Theorem~\ref{1Fconvex} and Remark~\ref{1Symm-Caratheodory}.
\end{remark}

\begin{theorem}\label{1FSymmYaxis}
The domain $\Omega_{\tau}:=\{z\in \mathbb{C}:1-\pi/4\leq \RE \tau(z)\leq 1+\pi/4\}$ displays symmetry relative to the line $\RE(w)=1$.
\end{theorem}
\begin{proof}
To verify the symmetry of the domain $\Omega_{\tau}$, it is only necessary to examine values of $\theta\in[0,\pi/2]$, as it is symmetrical about the real axis. The symmetry along the imaginary axis of a function $f\in \mathcal{A}$ is established for $\tau (z)=1+\arctan z$ when the real components of $\tau(\theta)$ and $\tau(\pi-\theta)$ sum to zero and the imaginary components of $\tau(\theta)$ and $\tau(\pi-\theta)$ are equal. 
Now, suppose $g(z)=\arctan(z)$. At $z=re^{it}$ for fixed $r\in (0,1)$,
\begin{equation*}
    \IM g(t)=\frac{1}{4}\log\bigg(\frac{1+r^2+2r\sin t}{1+r^2-2r\sin t}\bigg), \quad t\in [0,\pi].
\end{equation*}
We get the following expression, when $t\rightarrow \theta$ and $\pi-\theta$:
\begin{equation*}
    \IM g(\theta)=\dfrac{1}{4}\log\bigg(\dfrac{1+r^2+2r\sin\theta}{1+r^2-2r\sin\theta}\bigg)\quad\text{and}\quad  \IM g(\pi-\theta)=\dfrac{1}{4}\log\bigg(\dfrac{1+r^2+2r\sin(\pi-\theta)}{1+r^2-2r\sin(\pi-\theta)}\bigg).
\end{equation*}
Clearly, $\IM g(\theta)=\IM g(\pi-\theta).$ On the other hand, for $0\leq \theta\leq \pi/2$, 
\begin{align*}
\RE g(\theta)+\RE g(\pi-\theta)&=\dfrac{1}{2}\bigg\{\arg \bigg(\frac{1+ire^{i\theta}}{1-ire^{i\theta}}\bigg)+\arg \bigg(\frac{1+ire^{i(\pi-\theta)}}{1-ire^{i(\pi-\theta)}}\bigg)\bigg\}\\
&=0.
\end{align*}
Consequently, due to the translation property, $\tau(z)$ is symmetric about the real line $\RE(w)=1,$ while $g(z)$ is symmetric about the imaginary axis.
\end{proof}

\begin{remark}
Note that at $z=e^{i\theta}$, we have
\begin{equation}
   \lim_{\theta \rightarrow \pi/2} \IM \tau(z)= \lim_{\theta \rightarrow \pi/2}\bigg \{\frac{1}{4} \log \bigg(\frac{1+\sin\theta}{1-\sin\theta}\bigg)\bigg\}= \infty 
\end{equation}
and $1-\pi/4\leq \RE \tau(z)\leq 1+\pi/4$. That is,  real part of $\tau(z)$ is bounded whereas imaginary part of $\tau(z)$ is unbounded. Therefore, $\tau(z)$ maps the unit disk onto a {\bf{strip domain}} i.e., $\Omega_{\tau}$ is given by
\begin{equation*}
    \Omega_{\tau}:=\bigg\{z\in \mathbb{C}:1-\frac{\pi}{4}\leq \RE \tau(z)\leq 1+\frac{\pi}{4}\bigg\}.
\end{equation*}
\end{remark}

Thus, in light of the classes defined in \cite{goel,kumar-ganganiaCardioid-2021,mendi,sharma,raina,sokol},  we here below introduce the well-defined subclass of starlike functions associated with a vertical strip domain:
\begin{definition}
    Let $\tau(z)=1+ \arctan z$, which is univalent in $\mathbb{D}$. Then, the class of starlike functions associated with the inverse tangent function is defined as
    \begin{equation}\label{star}
		\mathcal {S}^{*}_{\tau}:=\bigg\{f\in \mathcal {A}:\dfrac{zf'(z)}{f(z)}\prec 1+\arctan z\bigg\}.
	\end{equation}
\end{definition}
 
The integral representation of the class defined in (\ref{star}) is described in the following definition:
\begin{definition}(Integral representation)
A function $f\in \mathcal {S}^{*}_{\tau}$ if and only if an analytic function $\psi(z)\prec \tau(z)$ exists such that 
\begin{equation}
f(z)=z \exp\int_{0}^{z}\dfrac{\psi(t)-1}{t}dt.\label{integral}
\end{equation}
		
By choosing $\psi(t)=\tau(t)$, in the expression given in equation (\ref{integral}), we obtain the function,
$$\tilde{\tau}(z)=z \exp\int_{0}^{z}\dfrac{\arctan t}{t}dt \in \mathcal {S}^{*}_{\tau}.$$
\end{definition}

Given Table \ref{table:1} shows some $\psi_{i} (i=1,2,3)$ functions which belong to the domain $\Omega_{\tau}$.
\begin{table}[h!]
\caption{Functions associated with the class $\mathcal{S}^{*}_{\tau}$}
\centering
\begin{tabular}{c c c } 
 \hline
 i & $\psi_i(z)$ & $f_i(z)$ \\ [0.5ex] 
 \hline
  1 &$1+z/2$& $z \exp(z/2)$ \\
  \hline
 2 &$1+\dfrac{z\exp(z/17)}{2}$ &  $z \exp\bigg(\frac{17}{2}(e^{z/17}-1)\bigg)$  \\
 \hline
 3&$1+\dfrac{z \sin z}{4}$&$z \exp\bigg(\dfrac{1-\cos z}{4}\bigg)$\\[1ex] 
 \hline
\end{tabular}
\label{table:1}
\end{table}	

And since $\tau(z)$ is univalent in $\mathbb{D}$, $\psi_i(\mathbb{D})\subset \tau(\mathbb{D})$ and $\psi_i(0)=\tau(0)$ are enough to prove that $\psi_i(z)\prec \tau(z).$ Thus, the functions $f_i(z)$ belong to the class $\mathcal {S}^{*}_{\tau}$ using the integral representation given in equation (\ref{integral}).
In particular, $\psi(z)=\tau(z)$ gives a function
\begin{equation}
\tilde{\tau}(z)=z \exp \int_{0}^{z}\dfrac{\arctan (t)}{t}dt=z+z^2+\frac{z^3}{2}+\frac{z^4}{18}-\frac{5}{72}z^5-\frac{13}{1800}z^6+\cdots,\label{1 extre}
\end{equation}
which plays the role of an extremal function for many problems involving $\mathcal{S}^{*}_{\tau}$. Note that $\tau(z)$ meets all the conditions of Ma-Minda function and therefore yields the following simple results:
\begin{theorem}[A]
Let $f\in \mathcal {S}^{*}_{\tau}$ and $\bar{\tau}(z)$ given by equation (\ref{1 extre}). Then the following results hold:
\begin{enumerate}
\item Subordination results: $\dfrac{zf'(z)}{f(z)}\prec \dfrac{z\bar{\tau}'(z)}{\bar{\tau}(z)}$ and $\dfrac{f(z)}{z}\prec \dfrac{\bar{\tau}(z)}{z}.$
\item Growth theorem: For $|z|=r<1, -\bar{\tau}(-r)\leq |f(z)|\leq \bar{\tau}(r).$
\item Covering theorem: Either f is a rotation of $\bar{\tau}$ or $\{w:|w|<-\bar{\tau}(-1)\}\subset f(\mathbb {D})$, where $-\bar{\tau}(-1) = \lim_{r \to 1} -\bar{\tau}(-r).$ 
\item Rotation theorem: For $|z|=r<1, \bigg|\arg \dfrac{f(z)}{z}\bigg|\leq \max_{|z|=r}\arg \dfrac{\bar{\tau}(z)}{z}.$
\end{enumerate} 
Equality holds in (2) and (4) for some non-zero z if and only if $f(z)$ is a rotation of $\bar{\tau}(z)$.
\end{theorem}

\begin{lemma}(Disk Containment)
\label{1 disklemma}
Let us consider the disk $D(a,r):=\{w\in \mathbb {C}:|w-a|=r\}$ with real $a$ and $1\in D(a,r),$ then we observe that $D(a,r)\in \Omega_{\tau}$ if and only if \begin{equation}
1-\frac{\pi}{4}\leq a-r<a+r\leq 1+\frac{\pi}{4}.\label{r}
\end{equation}
Moreover, the disk $D(1,\pi/4)$ is the largest disk contained in the region $\Omega_{\tau}$, further $D(a,r)\subset D(1,\pi/4)$ if and only if $$|a-1|\leq \frac{\pi}{4}-r.$$
It can also be viewed as,
\begin{equation}
	1-\frac{\pi}{4}\leq a-r\quad  \text{when}\quad  a\leq \frac{\pi}{4}\label{r1}
\end{equation}
and 
\begin{equation}
a+r\leq 1+\frac{\pi}{4}\quad \text{when}\quad a>\frac{\pi}{4}.\label{r2}
\end{equation}
\end{lemma}

\begin{figure}[h]
\centering
\includegraphics[width=3cm]{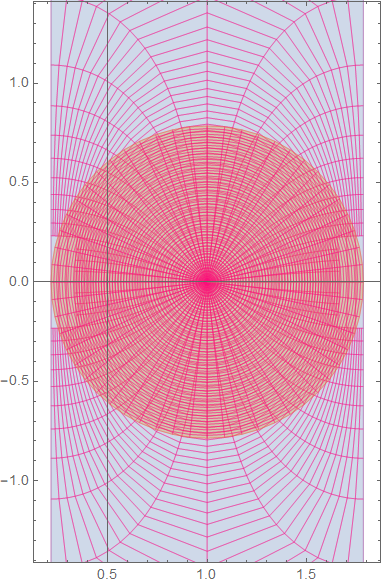}
\caption{The largest disk that can be inscribed inside the domain $\Omega_{\tau}$ is $D(1,\pi/4)$.}
\end{figure}

The proof of the Lemma \ref{1 disklemma} is skipped here as it directly follows from Theorem~\ref{1Fconvex}, Theorem~\ref{1FSymmYaxis} and Remark~\ref{1Symm-Caratheodory}.
\begin{corollary}
The disk $\{ w: |w-1|\leq \arctan(r)\}$ is contained in $\tau(|z|\leq r)$ and is maximal.
\end{corollary}

Next, we present a conclusion in the form of the following theorem.

\begin{theorem}
Let $-1<B<A\leq 1.$ Then, $p(z):=(1+Az)/(1+Bz) \in \mathcal{S}_{\tau}^{*}$ if and only if
\begin{equation}\label{1J-condition}
	1-\frac{\pi}{4}\leq \frac{1-A}{1-B}\leq \frac{1+A}{1+B}\leq 1+\frac{\pi}{4}, 
	\end{equation}
or equivalently,
\begin{equation}\label{1J-condition2}
	A\leq \begin{cases}
		\frac{\pi}{4}+(1-\frac{\pi}{4})B \quad \text{if} \quad \frac{1-AB}{1-B^2}\leq \frac{\pi}{4}\\
		\frac{\pi}{4}+(1+\frac{\pi}{4})B \quad \text{if} \quad \frac{1-AB}{1-B^2}>\frac{\pi}{4}. 
	\end{cases}
\end{equation}
\end{theorem}
\begin{proof} In order to prove that $p(z)\in \mathcal{S}_{\tau}^{*},$ we need to show that 
$p(z)\prec 1+\arctan z$ for $z\in \mathbb{D}.$
We note that the bilinear transformation $p(z)$ takes the unit disk $\mathbb {D}$ onto the disc $D(a,r)$ where $$a=\frac{1-AB}{1-B^2}, \quad r=\frac{A-B}{1-B^2}.$$
Then end points of the diameter of $D(a,r)$ are $a-r=(1-A)/(1-B)$ and $a+r=(1+A)/(1+B)$ such that $1\in D(a,r).$ Hence, $p\in \mathcal{S}_{\tau}^{*}$ if and only if $D(a,r)\subset \Omega_{\tau}$, and hence \eqref{1J-condition} follows from equation \eqref{r}. Equivalently, $p\in \mathcal{S}_{\tau}^{*}$ if and only if $D(a,r)\subset D(1,\pi/4)$, which proves \eqref{1J-condition2} with the help of \eqref{r1} and \eqref{r2}.
\end{proof}
\begin{corollary}
If $A$ and $B$ satisfy either \eqref{1J-condition} or $\eqref{1J-condition2}$, then $\mathcal{S}^{*}[A,B] \subset \mathcal {S}^{*}_{\tau}$.
\end{corollary}

\begin{example} The necessary and sufficient condition for $(1+Az)/(1+Bz)$ to be in $\mathcal {S}^{*}_{\tau}$ for some special choices of $A$ and $B$, is given by
\begin{itemize}
\item[(a)] When $B=0$, $1+Az\in \mathcal {S}^{*}_{\tau}$ if and only if $-\pi/4\leq A\leq \pi/4.$
\item[(b)] When $A=0$, $1/(1+Bz)\in \mathcal {S}^{*}_{\tau}$ if and only if $-\pi/(\pi+4)\leq B\leq \pi/(\pi+4).$
\item[(c)] When $B=-A$, $(1+Az)/(1-Az)\in \mathcal {S}^{*}_{\tau}$ if and only if $-\pi/(\pi+8)\leq A\leq \pi/(\pi+8).$ 
\end{itemize}
\end{example}

\section{Inclusion Relations and Certain Radius Estimates}
\noindent In this section, many inclusion properties of the class $\mathcal{S}^{*}_{\tau}$ involving the already well established classes, are derived. Let us now recall some subclasses of starlike functions in the present context, namely, the class of starlike functions of order $\alpha$ denoted by $\mathcal{S}^{*}(\alpha)$ $(0\leq \alpha <1)$ satisfying the condition, $\RE( zf'(z)/f(z))>\alpha.$ An interesting class studied by Uralegaddi et al.\cite{urale} is $\mathcal{M}(\alpha) (\alpha>1),$ satisfying $\RE(z f'(z)/f(z))<\alpha.$ Also, a closely related class studied by Ravichandran and Kumar\cite{kumar}, the class of starlike functions of reciprocal order $\alpha$ $(0\leq \alpha<1),$ denoted as $\mathcal{RS}^{*}(\alpha)$, satisfying $\RE(f(z)/(zf'(z)))>\alpha$ and the class of $k$-starlike functions denoted by $k-\mathcal{ST}$ $(k\geq 0)$, introduced by Kanas and Wisniowska \cite{kanas} with the characterization $\RE(zf'(z)/f(z))>k|zf'(z)/f(z)-1|$. In addition, by adding one more parameter $\alpha,$ Kanas and Raducanu generalised it to $\mathcal{ST}(k,\alpha)$ with the condition given by $\RE(zf'(z)/f(z))>k|zf'(z)/f(z)-1|+\alpha$.
\begin{theorem}
\label{theorem1}
The class $\mathcal{S}^{*}_{\tau}$ satisfies the following inclusion relations:
\begin{enumerate}
\item $\mathcal{S}^{*}_{\tau}\subset \mathcal{S}^{*}(\alpha),$ for $0\leq \alpha \leq 1-\pi/4.$
\item $\mathcal{S}^{*}_{\tau}\subset \mathcal{RS}^{*}(1/\alpha)\subset \mathcal{M}(\alpha),$ for $\alpha\geq 1+\pi/4.$
\item If $k>1$ and $0\leq \alpha<1,$ then $\mathcal{ST}(k,\alpha)\subset \mathcal{S}^{*}_{\tau},$ for $k\geq (\pi+4(1-\alpha))/\pi.$ In particular, $k-\mathcal{ST}\subset \mathcal{S}^{*}_{\tau}$ for $k\geq 1+4/\pi.$
\end{enumerate}
\end{theorem}
\begin{proof} 
(1) and (2) follows as under:\\
 Let $f\in \mathcal{S}^{*}_{\tau},$ then (\ref{star}) implies $zf'(z)/f(z)\prec 1+\arctan z.$ As we know that 
\begin{equation*}
\min_{|z|=1}\RE (1+\arctan z)< \RE\dfrac{zf'(z)}{f(z)}<\max_{|z|=1} \RE (1+\arctan z),
\end{equation*}
results in $$1-\frac{\pi}{4}<\RE \dfrac{zf'(z)}{f(z)}<1+\frac{\pi}{4}.$$
So, $f\in \mathcal{S}^{*}(1-\pi/4)\cap\mathcal{M}^{*}(1+\pi/4).$ Consider the inequality 
\begin{equation*}
\RE\dfrac{f(z)}{zf'(z)}>\min_{|z|=1}\RE \frac
{1}{1+\arctan z}=\frac{4}{4+\pi}.
\end{equation*}
It follows that $f\in \mathcal{RS}^{*}(\alpha)$ for $\alpha \leq 4/(4+\pi).$ Equivalently, $f\in \mathcal{RS}^{*}(1/{\alpha}),$ for $\alpha \geq 1+\pi/4.$ We know that $f\in \mathcal{RS}^{*}(1/{\alpha})$ implies that $\RE (zf'(z)/f(z))<\alpha.$ By combining the above facts, we get
\begin{equation*}
\mathcal{S}^{*}_{\tau}\subset \mathcal{RS}^{*}(1/\alpha)\subset \mathcal{M}(\alpha),\quad \text{for} \quad \alpha\geq 1+\pi/4.
\end{equation*}
(3) Let $\Delta_{k,\alpha}=\{w\in\mathbb{C}: \RE w>\alpha+k|w-1|\}$ be the domain whose boundary $\partial \Delta_{k,\alpha}$ represents an ellipse for $k>1,$ given by 
\begin{equation*}
    \dfrac{(x-x_0)^2}{a^2}+\dfrac{(y-y_0)^2}{b^2}=1,
\end{equation*}
with 
\begin{equation*}
    x_0=\dfrac{k^2-\alpha}{k^2-1},\quad  y_0=0,\quad  a=\bigg|\dfrac{k(\alpha-1)}{k^2-1}\bigg| \quad \text{and} \quad b=\bigg|\dfrac{\alpha-1}{\sqrt{k^2-1}}\bigg|.
\end{equation*}
We observe that $x_0+a$ should not exceed $1+\pi/4$ so that the ellipse $\Delta_{k,\alpha}$ lie inside the domain $\Omega_{\tau}$, which leads to the inequality $k\geq (\pi+4(1-\alpha))/\pi.$ In particular, when $\alpha=0,$ we have $k-\mathcal{ST}\subset \mathcal{S}^{*}_{\tau}$ whenever $k\geq 1+4/\pi.$
\end{proof}

The $\mathcal{C}_{\gamma}-$radius for the class $\mathcal{S}^{*}_{\tau}$ is obtained in the next theorem, where $\mathcal{C}_{\gamma}:=\mathcal{C}^{*}[1-2\gamma,-1]$ denotes the class of convex functions of order $\gamma\in [0,1)$.
\begin{theorem}
Let $f\in \mathcal{S}^{*}_{\tau}.$ Then $f\in \mathcal{C}_\gamma$ in $|z|<r_{\gamma},$ where $r_\gamma(\approx 0.387888)$ is the least positive root of the equation 
\begin{equation}
        (1-\arctan r)(1-r^4)(1-\arctan r-\gamma)-r=0, \quad (0\leq \gamma<1).\label{rgamma}
\end{equation}
\end{theorem}
\begin{proof}
Let $f\in \mathcal{S}^{*}_{\tau}$ and $w$ be a Schwarz function such that $zf'(z)/f(z)=1+\arctan(w(z)).$ Thus, we have
\begin{equation*}
    1+\dfrac{zf''(z)}{f'(z)}=1+\arctan(w(z))+\dfrac{zw'(z)}{(1+\arctan(w(z)))(1+w^2(z))},
\end{equation*}
which further implies
\begin{equation*}
    \RE \bigg(1+\dfrac{zf''(z)}{f'(z)}\bigg)\geq \RE(1+\arctan(w(z)))-\bigg|\dfrac{zw'(z)}{(1+\arctan(w(z)))(1+w^2(z))}\bigg|.
\end{equation*}
Since $w$ is a Schwarz function satisfying $|w'(z)|\leq (1-|w(z)|^2)/(1-|z|^2),$ we obtain
\begin{align*}
     \RE \bigg(1+\dfrac{zf''(z)}{f'(z)}\bigg)&\geq 1-\arctan(|z|)-\dfrac{|z|(1-|w(z)|^2)}{(1-\arctan(|z|))(1+|z|^2)(1-|z|^2)}\\
     &\geq 1-\arctan(|z|)-\dfrac{|z|}{(1-\arctan(|z|))(1-|z|^4)}.
\end{align*}

Let us consider $g(r):=1-\arctan r-r/((1-\arctan r)(1-r^4))$ with $g(0)=1$. It is easy to observe that $g$ is a decreasing function in $[0,1)$. Therefore, we deduce that $\RE(1+zf''(z)/f'(z))>\gamma$ in $|z|<r_\gamma<1,$ where $r_\gamma$ is the least positive root of $g(r)=\gamma$ as given in (\ref{rgamma}).
\end{proof}

In this theorem, we aim to find the best possible radius for which $\mathcal{S}^*_{L},\mathcal{S}^*_{C}, \mathcal{S}^*_{e},\Delta^{*}$ and $\mathcal{S}^*_{\wp}$ are contained in $\mathcal{S}_{\tau}^{*}$.

\begin{theorem}
The $\mathcal{S}^{*}_{\tau}-$radii for the classes $\mathcal{S}^{*}_{L}, \mathcal{S}^{*}_{C}, \mathcal{S}^{*}_{e}, \Delta^{*}, \mathcal{S}^{*}_{\wp}$ are as follows:
\begin{enumerate}
        \item $R_{S^*_{\tau}}(\mathcal{S}^{*}_{L})=\pi(8-\pi)/16.$
        \item $R_{S^*_{\tau}}(\mathcal{S}^{*}_{C})=\sqrt{1+3\pi/8}-1.$
        \item $R_{S^*_{\tau}}(\mathcal{S}^{*}_{e})=\ln(1+\pi/4).$
        \item $R_{S^*_{\tau}}(\Delta^{*})=\dfrac{\pi(8+\pi)}{8(4+ \pi)}.$
        \item $R_{S^*_{\tau}}(\mathcal{S}^{*}_{\wp})=0.484035.$
\end{enumerate}
All the above radii are sharp.
\end{theorem}
\begin{proof}
\begin{enumerate}
\item Let $f\in \mathcal{S}^*_{L}$, then we have $zf'(z)/f(z)\prec \sqrt{1+z}.$ For $|z|=r$  we obtain
\begin{equation*}
   \bigg|\dfrac{zf'(z)}{f(z)}-1\bigg|\leq 1-\sqrt{1-r}\leq \dfrac{\pi}{4},
\end{equation*}
which holds when $r\leq \pi(8-\pi)/16=R_{S^*_{\tau}}(\mathcal{S}^{*}_{L}).$ Consider the function
\begin{equation}
    \tilde{f}(z)=\dfrac{4z}{(\sqrt{1+z}+1)^2}e^{2(\sqrt{1+z}-1)}.
\end{equation}
We observe that $\tilde{f}\in \mathcal{S}^{*}_{L}$ as $z\tilde{f}'(z)/\tilde{f}(z)=\sqrt{1+z}$ and $z\tilde{f}'(z)/\tilde{f}(z)=1-\pi/4$ for $|z|=-R_{S^*_{\tau}}(\mathcal{S}^{*}_{L}),$ assuring the sharpness of the result. 

\item Let $f\in \mathcal{S}^*_{C}$. As $zf'(z)/f(z)\prec 2z^2/3+4z/3+1.$ For $|z|=r$, we obtain
\begin{equation*}
    \bigg|\dfrac{zf'(z)}{f(z)}-1\bigg|\leq \dfrac{2}{3}r^2+\dfrac{4}{3}r\leq \dfrac{\pi}{4}
\end{equation*}
which holds whenever $r\leq \sqrt{1+3\pi/8}-1=R_{S^*_{\tau}}(\mathcal{S}^{*}_{C}).$ Now, we define the function as 
\begin{equation*}
   \tilde{f}(z)=z\exp\bigg(\dfrac{z^2+4z}{3}\bigg).
\end{equation*}
Since $z\tilde{f}'(z)/\tilde{f}(z)=2z^2/3+4z/3+1, \tilde{f}\in \mathcal{S}^{*}_{C}.$ Further $z\tilde{f}'(z)/\tilde{f}(z)=1+\pi/4$ whenever $z=R_{S^*_{\tau}}(\mathcal{S}^{*}_{C}).$ Thus the result is sharp.

\item Let $f\in \mathcal{S}^*_{e}$. As $zf'(z)/f(z)\prec e^z.$ For $|z|=r$, we obtain
\begin{equation*}
    \bigg|\dfrac{zf'(z)}{f(z)}-1\bigg|\leq e^r-1\leq \dfrac{\pi}{4}
\end{equation*}
provided $r\leq \ln(1+\pi/4)=R_{S^*_{\tau}}(\mathcal{S}^{*}_{e}).$ Let
\begin{equation}
    \tilde{f}(z)=z\exp\bigg(\int_{0}^{z}\dfrac{e^{t}-1}{t}dt\bigg). \label{extremal2}
\end{equation}
Then, $\tilde{f}\in \mathcal{S}^{*}_{e}.$ Also, $z\tilde{f}'(z)/\tilde{f}(z)=1+\pi/4$ for $z=R_{S^*_{\tau}}(\mathcal{S}^{*}_{e}),$ establishing the sharpness of the result. 
\item Let $f\in \Delta^*$. As $zf'(z)/f(z)\prec z+\sqrt{1+z^2}.$ For $|z|=r$, we obtain
\begin{equation*}
 \bigg|\dfrac{zf'(z)}{f(z)}-1\bigg|\leq r+\sqrt{1+r^2}-1\leq \dfrac{\pi}{4}
\end{equation*}
provided $r\leq \pi(8+\pi)/8(4+ \pi)=R_{S^*_{\tau}}(\Delta^{*}).$ Now, consider the function 
\begin{equation}
    \tilde{f}(z)=\dfrac{2z}{\sqrt{1+z^2}+1}e^{(z+\sqrt{1+z^2}-1)}.
\end{equation}

Clearly $\tilde{f}\in \Delta^{*}.$ For $z=(\pi(8+\pi))/(8(4+ \pi)),$ the sharpness of the result is proved since $z\tilde{f}'(z)/\tilde{f}(z)=1+\pi/4.$ 
\item Let $f\in \mathcal{S}^*_{\wp}$. As $zf'(z)/f(z)\prec 1+ze^z.$ For $|z|=r$, we obtain
\begin{equation*}
    \bigg|\dfrac{zf'(z)}{f(z)}-1\bigg|\leq re^r\leq \dfrac{\pi}{4}
\end{equation*}
provided $r\leq 0.484035=R_{S^*_{\tau}}(\mathcal{S}^{*}_{\wp}),$ using a numerical computation. Let
\begin{equation}
    \tilde{f}(z)=z\exp (e^z-1).
\end{equation}
Then, $\tilde{f}\in \mathcal{S}^{*}_{\wp}.$ Also, $z\tilde{f}'(z)/\tilde{f}(z)=1+\pi/4$ for $z=R_{S^*_{\tau}}(\mathcal{S}^{*}_{\wp}),$ establishing the sharpness of the result.
\end{enumerate}
\end{proof}

In this theorem, we aim to find the best possible radius for which $\mathcal{S}_{\tau}^{*}$ is contained in $\mathcal{S}^*_{e},\mathcal{S}^*_{SG}, \mathcal{S}^*_{C},\mathcal{S}^*_{\wp}$ and $\Delta^{*}$. 
\begin{theorem}
The radii of $\mathcal{S}^{*}_{\tau}$ for the classes $  \mathcal{S}^{*}_{e},\mathcal{S}^{*}_{SG},\mathcal{S}^{*}_{C},\mathcal{S}^{*}_{\wp}, \Delta^{*},$ are as follows:
\begin{enumerate}
\item $R_{\mathcal{S}^*_{e}}(\mathcal{S}^{*}_{\tau})=\tan(1-1/e)\approx 0.732368.$
\item $R_{\mathcal{S}^*_{SG}}(\mathcal{S}^{*}_{\tau})=\tan((e-1)/(e+1))\approx0.498088.$
\item $R_{\mathcal{S}^*_{C}}(\mathcal{S}^{*}_{\tau})=\tan(2/3)\approx0.786843.$
\item $R_{\mathcal{S}^*_{\wp}}(\mathcal{S}^{*}_{\tau})=\tan(1/e)\approx0.385426.$
\item $R_{\Delta^{*}}(\mathcal{S}^{*}_{\tau})=\tan(2-\sqrt{2})\approx0.66347.$
\end{enumerate}
All the above radii are sharp.
\end{theorem}

\begin{proof}
The domains $\Omega_{\tau}$ and $q(\mathbb{D})$ for $q(z)=e^z$ and $2/(1+e^{-z})$ exhibit convexity as well as symmetry with respect to the real and imaginary axes. Considering the value of $r=R_{a}(\mathcal{S}^{*}_{\tau})$ for $a=\mathcal{S}^{*}_{e}$ and $\mathcal{S}^*_{SG}$, the result is sharp, as $\Omega_{\tau}$ touches the boundary of $\mathcal{S}^{*}_{e}$ at $z=0.367879$ and $\mathcal{S}^{*}_{SG}$ at $z=0.53789$, as can be seen in Figs. \ref{tau in e} and \ref{tau in sg}.

\begin{itemize}
\item[(1)] Let $f\in \mathcal{S}^*_{\tau}$. As $zf'(z)/f(z)\prec 1+\arctan z.$ From \cite{mendi}, the radius of the disk centered at 1 is $1-1/e.$ For $|z|=r$, we obtain
\begin{equation*}
    \bigg|\dfrac{zf'(z)}{f(z)}-1\bigg|\leq \arctan r\leq 1-\dfrac{1}{e}
\end{equation*}
which holds provided $r\leq \tan(1-1/e)=R_{S^*_{e}}(\mathcal{S}^{*}_{\tau}).$ 
\item[(2)] Let $f\in \mathcal{S}^*_{\tau}$ implies $zf'(z)/f(z)\prec 1+\arctan z.$ We observe that in \cite{goel}, the radius of the disk centered at 1 is $(e-1)/(e+1)$. For $|z|=r$, we obtain
\begin{equation*}
    \bigg|\dfrac{zf'(z)}{f(z)}-1\bigg|\leq \arctan r\leq \dfrac{2e}{1+e}-1
\end{equation*}
whenever $r\leq \tan((e-1)/(e+1))=R_{S^*_{SG}}(\mathcal{S}^{*}_{\tau}),$ using a numerical computation. 
\end{itemize}

\begin{figure}[!htb]
  \begin{minipage}{0.45\textwidth}
    \centering
     \includegraphics[width=.5\linewidth]{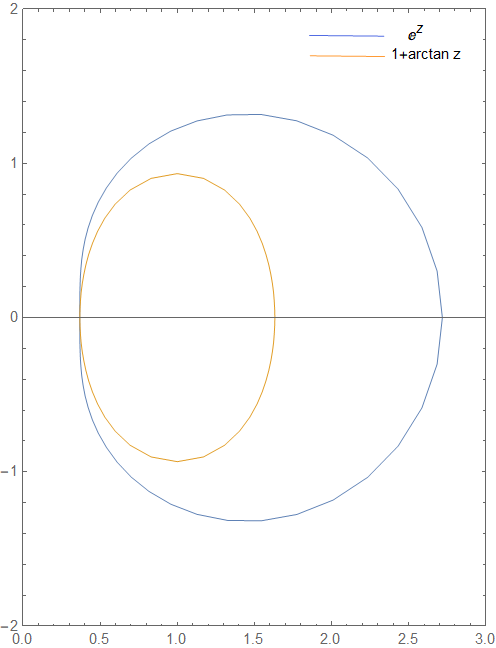}
    \caption{$\mathcal{S}_{\tau}^{*}\subset\mathcal{S}_{e}^{*} $}\label{tau in e}
  \end{minipage}\hfill
   \begin{minipage}{0.45\textwidth}
     \centering
     \includegraphics[width=.5\linewidth]{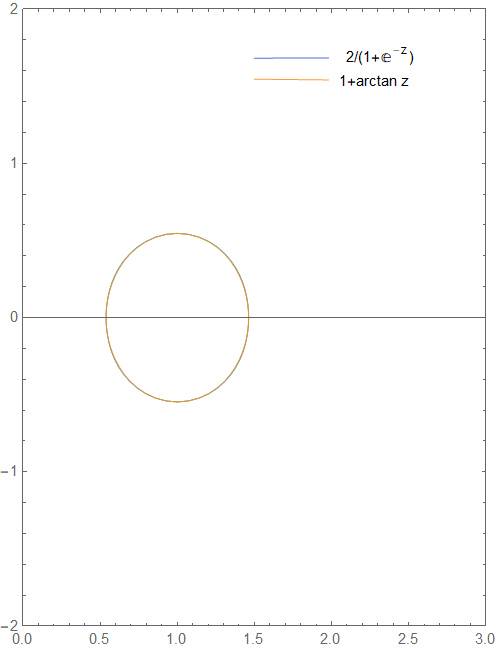}
     \caption{$\mathcal{S}_{\tau}^{*}\subset\mathcal{S}_{SG}^{*} $}\label{tau in sg}
   \end{minipage}
\end{figure}

Now, $t(\mathbb{D})$ (for $t(z)=1+4z/3+2z^2/3,1+ze^z, z+\sqrt{1+z^2}$) is convex in the direction of 1 and symmetric w.r.t real-axis. For the given $r=R_{b}(\mathcal{S}^{*}_{\tau})$ where $b=\mathcal{S}^*_{C}$, $\mathcal{S}^*_{\wp}$ and $\Delta^{*}$, note that $\Omega_{\tau}$ touches the boundary of
$\mathcal{S}^*_{C}$, $\mathcal{S}^*_{\wp}$ and $\Delta^{*}$ 
at $z=0.333333$, $0.63212$ and $z=0.414214$, see Fig. \ref{tau in c}, Fig. \ref{tau in cardioid} and Fig. \ref{tau in crescent} respectively, assuring the sharpness of the result.
\begin{itemize}
\item[(3)] Let $f\in \mathcal{S}^*_{\tau}$ gives $zf'(z)/f(z)\prec 1+\arctan z.$ From \cite{sharma}, we observe that the disk centered at $1$ in $\mathbb{D}(C)$ has radius $2/3.$ For $|z|=r$, we obtain
\begin{equation*}
    \bigg|\dfrac{zf'(z)}{f(z)}-1\bigg|\leq \arctan r\leq \dfrac{2}{3}
\end{equation*}
which holds whenever $r\leq \tan(2/3)=R_{S^*_{C}}(\mathcal{S}^{*}_{\tau}).$

\item[(4)] Let $f\in \mathcal{S}^*_{\tau}$, we have $zf'(z)/f(z)\prec 1+\arctan z.$ The radius of the disk centered at 1 is $1/e$ in \cite{kumar-ganganiaCardioid-2021}. For $|z|=r$, we obtain
\begin{equation*}
    \bigg|\dfrac{zf'(z)}{f(z)}-1\bigg|\leq \arctan r\leq \dfrac{1}{e}
\end{equation*}
provided $r\leq \tan(1/e)=R_{S^*_{\wp}}(\mathcal{S}^{*}_{\tau}).$ 

\item[(5)] Let $f\in \mathcal{S}^*_{\tau}$. As $zf'(z)/f(z)\prec 1+\arctan z.$ We see that $2-\sqrt{2}$ is the radius of the disk centered at 1 in \cite{raina}. For $|z|=r$, we get
\begin{equation*}
 \bigg|\dfrac{zf'(z)}{f(z)}-1\bigg|\leq \arctan r\leq 2-\sqrt{2}
\end{equation*}
for $r\leq \tan(2-\sqrt{2})=R_{\Delta^{*}}(\mathcal{S}^{*}_{\tau}).$
\begin{figure}[!htb]
\minipage{0.25\textwidth}
  \includegraphics[width=\linewidth]{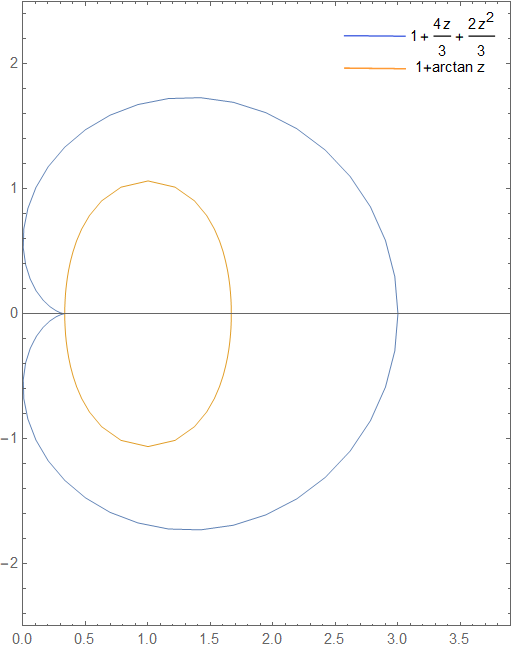}
  \caption{$\mathcal{S}_{\tau}^{*}\subset\mathcal{S}_{C}^{*} $}\label{tau in c}
\endminipage\hfill
\minipage{0.25\textwidth}
  \includegraphics[width=\linewidth]{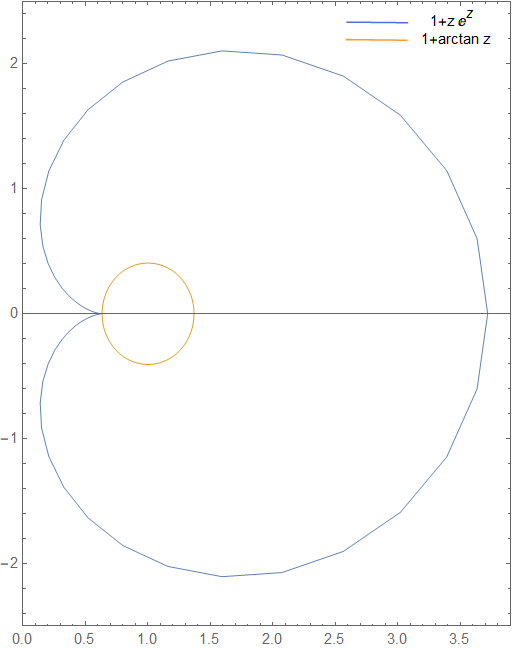}
  \caption{$\mathcal{S}_{\tau}^{*}\subset\mathcal{S}_{\wp}^{*} $}\label{tau in cardioid}
\endminipage\hfill
\minipage{0.25\textwidth}%
  \includegraphics[width=\linewidth]{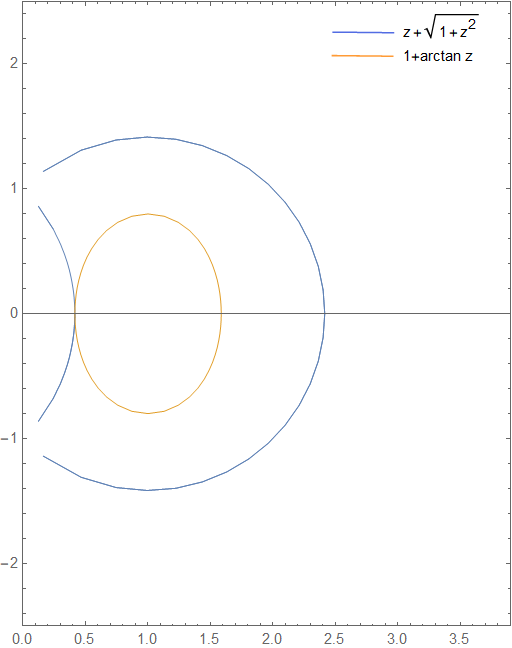}
  \caption{$\mathcal{S}_{\tau}^{*}\subset\Delta^{*} $}\label{tau in crescent}
\endminipage
\end{figure}
\end{itemize}
\end{proof}

\section{Sharp Coefficient Bounds and Hankel Determinants}
\noindent The determination of coefficient bounds for functions belonging to a given class is a significant problem in geometric function theory, see \cite{goodman vol1}, as it has a considerable influence on their geometric properties. For example, the bound on the second coefficient determines the growth and distortion characteristics of the class $\mathcal{S}$. Therefore, in this section, we examine various coefficient estimates for functions within the class $\mathcal {S}^{*}_{\tau}$. The results presented below are necessary for proving our main result.

\begin{lemma}\cite{rj}
Let $p\in \mathcal{P}$ be of the form $1+\sum_{n=1}^{\infty}p_nz^n.$ Then
\begin{equation}
        |p_1^4-3p_1^2p_2+p_2^2+2p_1p_3-p_4|\leq 2\label{p}
\end{equation}
and
\begin{equation}
       |p_3-2p_1p_2+p_1^3|\leq 2.\label{q}
    \end{equation}
    \end{lemma}
\begin{lemma}\cite{ma-minda}
Let $p\in \mathcal{P}$ be of the form $1+\sum_{n=1}^{\infty}p_nz^n.$ Then
\begin{align*}
      |p_2-\beta p_1^2|\leq \begin{cases} 2-4\beta,& \beta\leq 0;\\
      2, & 0\leq \beta\leq 1;\\
      4\beta-2, & \beta\geq 1
      \end{cases}
\end{align*}
when $\beta<0$ or $\beta>1,$ the equality holds if and only if $p(z)=(1+z)/(1-z)$ or one of its rotations. If $0<\beta<1,$ then the inequality holds if and only if $p(z)=(1+z^2)/(1-z^2)$ or one of its rotations. If $\beta=0,$ the equality holds if and only if $p(z)=(1+\eta)(1+z)/(2(1-z))+(1-\eta)(1-z)/(2(1+z))(0\leq \eta\leq 1)$ or one of its rotations. If $\beta=1,$ the equality holds if and only if p is the reciprocal of one of the functions such that the equality holds in case of $\beta=0.$ Though the above upper bound is sharp for $0<\beta<1,$ still it can be improved as follows:
\begin{equation}
      |p_2-\beta p_1^2|+\beta|p_1|^2\leq 2 \quad (0<\beta\leq 1/2)\label{use}
\end{equation}
and 
\begin{equation*}
      |p_2-\beta p_1^2|+(1-\beta)|p_1|^2\leq 2 \quad (1/2<\beta\leq 1).
\end{equation*}
\end{lemma}
Also, we recall that 
\begin{equation}
    \max_{0\leq t\leq 4}(At^2+Bt+C)=
    \begin{cases}
    C,& B\leq0, A\leq \frac{-B}{4};\\
    16A+4B+C, & B\geq0, A\geq \frac{-B}{8}\quad \text{or}\quad B\leq 0, A\geq \frac{-B}{4};\\
    \dfrac{4AC-B^2}{4A},& B>0, A\leq \frac{-B}{8}.\label{l}
    \end{cases}
\end{equation}

\begin{theorem}
Let $f(z)=z+\sum_{n=2}^{n=\infty}a_nz^n \in \mathcal {S}^{*}_{\tau},$ then 
$ |a_2|\leq 1,
		  |a_3|\leq 1/2,
		|a_4|\leq 1/3,$ and
		$|a_5|\leq 323/528\approx 0.611742.$
All these bounds are sharp.
\end{theorem}
\begin{proof} 
Let $f\in \mathcal {S}^{*}_{\tau},$ then there exists a Schwarz function $w(z)=\sum_{k=1}^{\infty}w_{k}z^k$ such that 
\begin{equation}
	\dfrac{zf\;'(z)}{f(z)}= 1+\arctan (w(z)).\label{schwarz}
	\end{equation}
Assume that $w(z)=(p(z)-1)/(p(z)+1)$ with $p(z)=1+p_{1}z+p_{2}z^2+\cdots \in \mathcal {P}.$ After substituting the values of $w(z)$, $p(z)$ and $f(z)$ in \eqref{schwarz} and comparing the corresponding coefficients, we get 
\begin{equation}
a_2=\frac{p_1}{2}, \quad a_3=\frac{p_2}{4}, \quad a_4=\frac{-1}{12}\bigg(\frac{{p_{1}}^3}{6}+\frac{p_{1}p_2}{2}-2p_3\bigg)\label{1 a2}
\end{equation}
\begin{equation}
\text{and} \quad
		a_5=\frac{-1}{8}\bigg(\frac{-5p_{1}^4}{72}+\frac{p_{2}^2}{4}+\frac{p_{1}p_{3}}{3}+\frac{p_{1}^2p_2}{6}-p_4\bigg). \label{a5}
	\end{equation}
 By using the previously established coefficient estimate for the Carath\'{eodory} class of functions, where $|p_n|\leq 2$ for $(n\geq 1)$, we are able to derive the bounds for $a_2$ and $a_3$. Specifically, we obtain $|a_2|\leq 1$ and $|a_3|\leq 1/2$.\\
For $a_4,$ equation (\ref{schwarz}) is re-written as:
\begin{equation}
		zf\;'(z)=[1+\arctan (w(z))]f(z). \label{re}
\end{equation}
On substituting $f(z)=z+\sum_{n=2}^{n=\infty}a_nz^n$ and $w(z)=\sum_{k=1}^{\infty}w_{k}z^k$ in equation (\ref{re}) and comparing the coefficients of $z^4,$ we get $$3a_4=w_3+\frac{3}{2}w_1w_2+\frac{1}{6}{w_1}^3.$$
By applying \cite[Lemma 2]{poko}, we establish that $|a_4|\leq 1/3$.\\
Now, using \cite[Theorem 4.15]{spkumar}, we have 
\begin{align*}
	a_5&=\frac{-1}{8}\bigg(\frac{-5}{72}{p_{1}}^4+\frac{1}{4}{p_{2}}^2+\frac{1}{3}p_{1}p_{3}+\frac{1}{6}{p_{1}^2}p_2-p_4\bigg)\\
	&=\dfrac{-1}{8}\bigg(\dfrac{1}{4}P-\dfrac{1}{6}p_1Q+\dfrac{7}{12}p_1^2R-\dfrac{3}{4}p_4\bigg)\\
	&\leq \dfrac{1}{8}\bigg(\dfrac{1}{4}|P|+\dfrac{1}{3}|Q|+\dfrac{7}{12}|p_1|^2|R|+\dfrac{3}{2}\bigg),
\end{align*}
with $P=p_1^4+p_2^2-3p_1^2p_2+2p_1p_3-p_4$, $Q=p_1^3-2p_1p_2+p_3$ and $R=p_2-(11/42)p_1^2$. Now from equations (\ref{p}), (\ref{q}) and (\ref{use}), we have $|P|\leq 2$, $|Q|\leq 2$ and $|R|\leq 2$ respectively. Upon substituting these values in the expression of $a_5$ above, we get
\begin{equation*}
		    |a_5|\leq \dfrac{1}{8}\bigg(\dfrac{8}{3}+\dfrac{7}{6}|p_1|^2-\dfrac{11}{72}|p_1|^4\bigg).
\end{equation*}
Now, we obtain $|7|p_1|^2/6-11|p_1|^4/72|\leq 49/22$ using the equation (\ref{l}) by taking $A=-11/72$, $B=7/6$ and $C=0,$ which leads to the desired estimate for $a_5.$  The sharpness of the result can be witnessed when $p_1=p_4=2$, $p_2=(1/33)(-44-i\sqrt{22})$ and $p_3=-2$.\\
The function 
	\begin{equation*}
	 f_n(z)=z\exp\bigg(\int_{0}^{z}\frac{\arctan (t^{n-1})}{t}dt\bigg)  
	\end{equation*}
acts as the extremal function  for the initial coefficients $a_n$ for $n=2,3$ and $4$.
\end{proof}	

	
\subsection{Sharp Hankel Determinants}
In 1966, Pommerenke \cite{pomi} introduced the concept of the Hankel determinants for the class $\mathcal{S}$ and later it was studied by others as well. The expression of $q^{th}$ Hankel determinant for a function $f\in \mathcal{A}$, whose coefficients are given in (\ref{form}), is as follows:
\begin{equation*}
	H_{q}(n) =\begin{vmatrix}
a_n&a_{n+1}& \ldots &a_{n+q-1}\\
a_{n+1}&a_{n+2}&\ldots &a_{n+q}\\
\vdots& \vdots &\ddots &\vdots\\
a_{n+q-1}&a_{n+q}&\ldots &a_{n+2q-2}
\end{vmatrix}\quad \text{for} \quad q,n\in \mathbb{N}.
\end{equation*}
For some special choices of $n$ and $q$; $a_2a_3-a_4$ is a type of famous Fekete-Szeg\"{o} functional and $H_2(2)=a_2a_4-a_3^2$ is second order Hankel determinant. While, assuming $a_1:=1$, the expression of third order Hankel determinant is given by
\begin{equation}
	H_{3}(1) =\begin{vmatrix}
1&a_{2} &a_{3}\\
a_{2}&a_{3}&a_{4}\\
a_{3}&a_{4} &a_{5}
\end{vmatrix}
=a_3(a_2a_4-a_3^2)-a_4(a_4-a_2a_3)+a_5(a_3-a_2^2).\label{1h3}
\end{equation}		
The following lemma serves as a basis for establishing our main result as it contains the well-known formula for $p_2$ and $p_3$ and $p_4$.
\begin{lemma}\cite{rj,lemma1}\label{b2b3b4}
Let $p\in \mathcal {P}$ of the form $1+\sum_{n=1}^{\infty}p_n z^n.$ Then
	\begin{equation*}
		2 p_2=p_1^2+\gamma (4-p_1^2),
	\end{equation*}
\begin{equation*}
	4p_3=p_1^3+2p_1(4-p_1^2)\gamma -p_1(4-p_1^2) {\gamma}^2+2(4-p_1^2)(1-|\gamma|^2)\eta
\end{equation*}
and \begin{equation*}
	8p_4=p_1^4+(4-p_1^2)\gamma (p_1^2({\gamma}^2-3\gamma+3)+4\gamma)-4(4-p_1^2)(1-|\gamma|^2)(p_1(\gamma-1)\eta+\bar{\gamma}{\eta}^2-(1-|\eta|^2)\rho),
\end{equation*}
for some $\gamma$, $\eta$ and $\rho$ such that $|\gamma|\leq 1$,  $|\eta|\leq 1$ and $|\rho|\leq 1.$
\end{lemma}

\begin{theorem}
Let $f\in \mathcal {S}^{*}_{\tau}.$ Then
\begin{equation*}
        |a_2a_3-a_4|\leq \dfrac{1}{3}.
\end{equation*}
The result is sharp.
\end{theorem}

 Let the function $f:\mathbb{D}\rightarrow \mathbb{C}$ be defined as
\begin{equation}
f(z)=z\exp\bigg(\int_{0}^{z}\dfrac{\arctan(t^3)}{t}dt\bigg)=z+\dfrac{z^4}{3}+\dfrac{z^7}{18}+\cdots,\label{extremal}
\end{equation}
with $f(0)=0$ and $f'(0)=1$. The function defined in equation (\ref{extremal}) acts as an extremal function for the bounds of $|a_2a_3-a_4|$ for the values of $a_2=a_3=0$ and $a_4=1/3.$

We now establish the bound for $H_2(2)$ in the following result by skipping the proof.
\begin{theorem}
Let $f\in \mathcal {S}^{*}_{\tau}.$ Then
\begin{equation*}
        |H_{2}(2)|\leq \dfrac{1}{4}.
\end{equation*}
The result is sharp.
\end{theorem}

For the sharpness of the above result, we consider the function $f:\mathbb{D}\rightarrow \mathbb{C}$ be defined as
\begin{equation*}
f(z)=z\exp\bigg(\int_{0}^{z}\dfrac{\arctan(t^2)}{t}dt\bigg)=z+\dfrac{z^3}{2}+\dfrac{z^5}{8}+\cdots,
\end{equation*}
with $f(0)=0$ and $f'(0)=1$. 

Obtaining the upper bound for the third order Hankel determinant is a challenging task, especially if we aim to find a sharp result as noted in \cite{kumar-ganganiaCardioid-2021,sharp}. Thus, finding the sharp bound for the third order Hankel determinant in \cite{kumar-ganganiaCardioid-2021} was an open problem and has recently been solved in \cite{nehacardioid}. Motivated by the works in \cite{sharp,nehacardioid,nehaalpha}, we now establish the sharp upper bound for the third order Hankel determinant for functions belonging to the class $\mathcal{S}^{*}_{\tau}$.


\begin{theorem}
Let $f\in \mathcal {S}^{*}_{\tau}.$ Then 
\begin{equation}
	|H_3(1)|\leq \dfrac{1}{9}.\label{9.5}
\end{equation}
The result is sharp.
\end{theorem}

\begin{proof} 
Since the class $\mathcal {P}$ is invariant under rotation, the value of $p_1=:p$ lies in the interval [0,2]. Then, substituting the values of $a_i's$ $(i=2,3,4,5)$ from equations (\ref{1 a2}) and (\ref{a5}) in the equation (\ref{1h3}), we get,
$$H_3(1)=\dfrac{1}{20736}\bigg(-49p^6+57p^4p_2-198p^2{p_2}^2-486{p_2}^3+312p^3p_3+936p p_2 p_3-576p_3^2-648p^2p_4+648p_2p_4\bigg).$$
On simplification using Lemma \ref{b2b3b4}, we obtain
$$H_3(1)=\dfrac{1}{82944}\bigg(\mu_1(p,\gamma)+\mu_2(p,\gamma)\eta+\mu_3(p,\gamma){\eta}^2+\psi(p,\gamma,\eta)\rho\bigg),$$
where $\gamma,\eta,\rho\in \mathbb {D}$ and
\begin{align*}	
\mu_1(p,\gamma):&=-49p^6-81{\gamma}^2p^2(4-p^2)^2-648{\gamma}^2p^2(4-p^2)-135{\gamma}^3p^2(4-p^2)^2+18{\gamma}^4p^2(4-p^2)^2\\
&\quad+117{\gamma}p^4(4-p^2)+324{\gamma}^3(4-p^2)^2-6p^4{\gamma}^2(4-p^2)-162p^4{\gamma}^3(4-p^2),\\
\mu_2(p,\gamma):&=(1-|\gamma|^2)(4-p^2)\bigg(336p^3+648{\gamma}p^3+(4-p^2)(432{\gamma}p-72p\gamma^2)\bigg),\\
\mu_3(p,\gamma):&=72(1-|\gamma|^2)(4-p^2)\bigg(-(8+|\gamma|^2)(4-p^2)+9p^2\bar{\gamma}\bigg),\\
\psi(p,\gamma,\eta):&=648(1-|\gamma|^2)(4-p^2)(1-|\eta|^2)\bigg(-p^2+\gamma(4-p^2)\bigg).
\end{align*}
Further, by taking $x=|\gamma|$, $y=|\eta|$ and using the fact that $|\rho|\leq 1,$ we have
\begin{align*}
	|H_3(1)|\leq \dfrac{1}{82944}\bigg(|\mu_1(p,\gamma)|+|\mu_2(p,\gamma)|y+|\mu_3(p,\gamma)|y^2+|\psi(p,\gamma,\eta)|\bigg)\leq H(p,x,y),
\end{align*}
where 
\begin{equation}
		H(p,x,y)=\dfrac{1}{82944}\bigg(h_1(p,x)+h_2(p,x)y+h_3(p,x)y^2+h_4(p,x)(1-y^2)\bigg)\label{new}
\end{equation} 
with 
\begin{align*}
	h_1(p,x):&=49p^6+81x^2p^2(4-p^2)^2+135x^3p^2(4-p^2)^2+18x^4p^2(4-p^2)^2+324x^3(4-p^2)^2\\
	&\quad +6x^2p^4(4-p^2)+117xp^4(4-p^2)+648x^2p^2(4-p^2)+162x^3p^4(4-p^2),\\
h_2(p,x):&=(1-x^2)(4-p^2)\bigg(336p^3+648p^3x+(4-p^2)(432px+72px^2)\bigg),\\
h_3(p,x):&=72(1-x^2)(4-p^2)\bigg((8+x^2)(4-p^2)+9p^2x\bigg),\\
h_4(p,x):&=648(1-x^2)(4-p^2)\bigg(p^2+x(4-p^2)\bigg).
	\end{align*}

In the closed cuboid $S:[0,2]\times [0,1]\times [0,1]$, we must now maximise $H(p,x,y)$. By locating the maximum values inside the six faces, on the twelve edges, and inside $S$, we can prove this.
\begin{enumerate}
\item Initially, we consider all the interior points of $S$. Let $(p,x,y)\in (0,2)\times (0,1)\times (0,1).$ In order to find the points of maxima in the interior of $S$, we partially differentiate equation (\ref{new}) with respect to $y$. We get
\begin{align*}
	    \dfrac{\partial H}{\partial y}&=\dfrac{(4 - p^2) (1 - x^2)}{3456} \bigg(12 p x (6 + x) + p^3 (14 + 9 x - 3 x^2) - 6 p^2 (17 - 18 x + x^2) y\\
	    & \quad \quad \quad \quad\quad\quad\quad\quad\quad+24 (8 - 9 x + x^2) y\bigg). 
\end{align*}
Now $\partial H/\partial y=0$ gives
\begin{equation*}
	    y=y_0:=\dfrac{12 p x (6 + x) + p^3 (14 + 9 x - 3 x^2)}{6 (1 - x)^2 (1 + x) (4 - p^2) (p^2 (17 - x) - 4 (8- x))}.
\end{equation*}
For the existence of critical points, $y_0$ should belong to $(0,1),$ which is possible only when 
\begin{align}
12 p x (6 + x) + p^3 (14 + 9 x - 3 x^2) &< 
 6 (4 - p^2)(1 - x)^2 (1 + x) \bigg(p^2 (17 - x)\nonumber\\
 &\quad- 4 (8 - x)\bigg) .\label{h1}
\end{align}

Now, we find the solutions satisfying the inequality (\ref{h1}) 
for the existence of critical points using hit and trial method. If we assume $p$ tends to 0 and 2, then there does not exist any $x\in (0,1)$ satisfying the equation (\ref{h1}). Similarly, if we assume $x$ tends to 0, then equation (\ref{h1}) holds for all $p>1.45018$. 
On calculations, we observe that there does not exist any $x\in (0,1)$ when $p\in (0,1.45018).$ If $x$ tends to 1, then no such $p\in (0,2)$ exists. Thus, we conclude that the function $G$ has no critical point in $(0,2)\times (0,1)\times (0,1).$ 

\item Now, we consider the interior of all the six faces of the cuboid $S$.\\
\noindent \underline{On $p=0$} 
\begin{equation}
g_1(x,y):=\dfrac{1}{144}\bigg(9x^3+2(1-x^2)(8+x^2)y^2+18x(1-x^2)(1-y^2)\bigg),\label{9.4}
\end{equation}
with $x,y\in (0,1)$. We observe that $g_1$ has no critical points in $(0,1)\times(0,1)$ as 
\begin{equation*}
    \dfrac{\partial g_1}{\partial y}=\dfrac{y(x-8)(1-x^2)(x-1)}{36}\neq 0\quad x,y\in (0,1).
\end{equation*}
\noindent \underline{On $p=2$} 
\begin{equation}
    H(2,x,y):=\dfrac{49}{1296},\quad x,y\in (0,1).\label{9.3}
\end{equation}
\noindent \underline{On $x=0$}
\begin{equation}
    g_2(p,y):=H(p,0,y)=\dfrac{49p^6+24(4-p^2) (27 p^2 + 14 p^3 y + 96 y^2 - 51 p^2 y^2)}{82944}\label{9.1}
\end{equation}
with $p\in (0,2)$ and $y\in (0,1).$ We solve $\partial g_2/\partial p$ and $\partial g_2/\partial y$ to locate the points of maxima. On solving $\partial g_2/\partial y=0,$ we obtain
\begin{equation}
    y=-\dfrac{7p^3}{3(32-17p^2)}(=:y_1).\label{y}
\end{equation}
For the given range of $y$, we should have $y_1\in (0,1),$ which is possible only if $p>p_0\approx 1.54572.$ On computations, $\partial g_2/\partial p=0$ gives
\begin{equation}
    864 p - 432 p^3 + 49 p^5 + 672 p^2 y - 280 p^4 y - 2400 p y^2 + 
 816 p^3 y^2=0.\label{9}
\end{equation}
After substituting equation (\ref{y}) in equation (\ref{9}), we get,
\begin{equation}
    1327104 - 2073600 p^2 + 1079568 p^4 - 215496 p^6 + 5243 p^8=0.\label{40}
\end{equation}
A numerical computation shows that the solution of (\ref{40}) in the interval $(0,2)$ is $p\approx 1.39637.$
Thus $g_2$ has no critical point in $(0,2)\times (0,1).$\\
\noindent \underline{On $x=1$}
\begin{equation}
    g_3(p,y):=H(p,1,y)=\dfrac{(2592+1872p^2-528p^4-p^6)}{41472}, \quad p\in (0,2).\label{9.2}
\end{equation}
and $p=p_0:\approx 1.32811$ comes out to be the critical point while computing $\partial g_3/\partial p=0.$ Undergoing simple calculations, $g_3$ achieves its maximum value $\approx 0.102376$ at $p_0.$\\
\noindent \underline{On $y=0$}
\begin{align*}
    g_4(p,x):=H(p,x,0)&=\dfrac{1}{82944}\bigg(p^6 (49 - 117 x + 75 x^2 - 27 x^3 + 18 x^4)+5184x(2-x^2)\\
    &\quad \quad\quad\quad\quad +144 p^2 (18 - 36 x + 9 x^2 + 33 x^3 + 2 x^4)\\
    &\quad \quad\quad\quad\quad-12 p^4 (54 - 93 x + 52 x^2 + 63 x^3 + 12 x^4)\bigg).
\end{align*}
On computations,
\begin{align*}
    \dfrac{\partial g_4}{\partial x}&=\dfrac{1}{82944}\bigg(p^6 (-117 + 150 x - 81 x^2 + 72 x^3)-10368 x^2 + 5184 (2- x^2)\\
    &\quad \quad\quad\quad\quad +144 p^2 (-36 + 18 x+ 99 x^2 + 8 x^3) -12 p^4 (-93 + 104 x\\
   &\quad \quad\quad\quad\quad + 189 x^2 + 48 x^3)\bigg)
\end{align*}
and \begin{align*}
    \dfrac{\partial g_4}{\partial p}&=\dfrac{p}{13824}\bigg(p^4 (49 - 117 x + 75 x^2 - 27 x^3 + 18 x^4)+48 (18 - 36 x \\
    &\quad \quad\quad\quad\quad+ 9 x^2 + 33 x^3 + 2 x^4)-8 p^2 (54 - 93 x + 52 x^2 + 63 x^3 + 12 x^4) \bigg).
\end{align*}
A numerical computation shows that there does not exist any solution for the system of equations $\partial g_4/\partial x=0$ and $\partial g_4/\partial p=0$ in $(0,2)\times (0,1).$\\
\noindent \underline{On $y=1$}
\begin{align*}
    g_5(p,x)&=\frac{1}{82944}\bigg( p^6 (49 - 117 x + 75 x^2 - 27 x^3 + 18 x^4)+1152 p x (6 + x - 6 x^2 - x^3)\\
&\quad\quad\quad\quad\quad +96 p^3 (14 - 9 x - 20 x^2 + 9 x^3 + 6 x^4) - 
 12 p^4 (-48 + 15 x + 148 x^2 \\
&\quad\quad\quad\quad\quad - 45 x^3 + 18 x^4)+144 p^2 (-32 + 18 x + 55 x^2 - 21 x^3 + 6 x^4)\\
&\quad\quad\quad\quad\quad-576 (-16 + 14 x^2 - 9 x^3 + 2 x^4)\\
&\quad\quad\quad\quad\quad-24 p^5 (14 + 9 x - 17 x^2 - 9 x^3 + 3 x^4)\bigg).
\end{align*}

\item Now, we calculate the maximum values achieved by $H(p,x,y)$ on the edges of cuboid $S$.\\ 
From equation (\ref{9.1}), we have $H(p,0,0)=s_1(p):=(49p^6+2592p^2-648p^4)/82944.$ It is easy to observe that $s_1'(p)=0$ for $p=\alpha_0:=0$ and $p=\alpha_1:=1.75123$ in the interval $[0,2]$ as points of minima and maxima respectively. The maximum value of $s_1(p)$ is $\approx 0.0393988.$ Hence, 
\begin{equation*}
    H(p,0,0)\leq 0.0393988.
\end{equation*}
Now considering the equation (\ref{9.1}) at $y=1,$ we get $H(p,0,1)=s_2(p):=(9216-4608p^2+1344p^3+576p^4-336p^5+49p^6)/82944.$ It is easy to observe that $s_2'(p)$ is a decreasing function in $[0,2]$ and hence $p=0$ acts as its point of maxima. Thus
\begin{equation*}
    H(p,0,1)\leq \dfrac{1}{9}, \quad p\in [0,2].
    \end{equation*}
Through computations, equation (\ref{9.1}) shows that $H(0,0,y)$ attains its maximum value at $y=1.$ This implies that
\begin{equation*}
    H(0,0,y)\leq \dfrac{1}{9}, \quad y\in [0,1].
\end{equation*}
Since, the equation (\ref{9.2}) is independent of $x$, we have $H(p,1,1)=H(p,1,0)=s_3(p):=(2592 + 1872 p^2-528 p^4 - p^6)/41472.$ Now, $s_3'(p)=624-352p^3-p^5=0$ when $p=\alpha_2:=0$ and $p=\alpha_3:=1.32811$ in the interval $[0,2]$ with $\alpha_2$ and $\alpha_3$ as points of minima and maxima respectively. Hence
\begin{equation*}
    H(p,1,1)=H(p,1,0)\leq 0.102376,\quad p\in [0,2].
\end{equation*}
On substituting $p=0$ in equation (\ref{9.2}), we get, $H(0,1,y)=1/16.$\\
The equation (\ref{9.3}) is independent of the all the variables namely $p$, $x$ and $y$. Thus, the maximum value of $H(p,x,y)$ on the edges $p=2, x=1; p=2, x=0; p=2, y=0$ and $p=2, y=1,$ respectively, is given by
\begin{equation*}
    H(2,1,y)=H(2,0,y)=H(2,x,0)=H(2,x,1)=\dfrac{1}{16},\quad x,y\in [0,1].
\end{equation*}
Using equation (\ref{9.4}), we obtain $H(0,x,1)=s_4(x):=(16-14x^2+9x^3-2x^4)/144.$ Upon calculations, we see that $s_4$ is a decreasing function in $[0,1]$ and hence attains its maximum value at $x=0.$ Hence
 \begin{equation*}
     H(0,x,1)\leq \dfrac{1}{9},\quad x\in [0,1].
 \end{equation*}
On again using equation (\ref{9.4}), we get $H(0,x,0)=s_5(x):=x(2-x^2)/16.$ On further calculations, we get $s_5'(x)=0$ for $x=x_0:=\sqrt{2}/\sqrt{3}.$ Also, $s_5(x)$ is an increasing function in $[0,x_0)$ and decreasing in $(x_0,1].$ So, it attains its maximum value at $x_0.$ Thus
 \begin{equation*}
     H(0,x,0)\leq 0.0680413,\quad x\in [0,1].
 \end{equation*}
\end{enumerate}
In view of all the cases, the inequality (\ref{9.5}) holds.\\
The function defined in equation (\ref{extremal}) acts as an extremal function for the bounds of $|H_3(1)|$ for the values of $a_2=a_3=a_5=0$ and $a_4=1/3.$
\end{proof}

	%
\subsection*{Acknowledgment}
Neha Verma is thankful to the Department of Applied Mathematics, Delhi Technological University, New Delhi-110042 for providing Research Fellowship.

\end{document}